\documentclass[a4paper,twoside,11pt]{article}

\usepackage{a4wide} 
\usepackage[utf8]{inputenc}
\usepackage[T1]{fontenc}
\usepackage{lmodern}

\usepackage{graphicx} 
\usepackage{verbatim} 
\usepackage{amsmath,amssymb,amsthm}
\usepackage{color}
\usepackage{enumerate}
\usepackage{array}
\usepackage{cite}
\usepackage{multirow}

\usepackage{caption}
\usepackage{subfig}

\usepackage{algorithm}
\usepackage{algpseudocode}

\usepackage{tikz}

\usepackage{titling}
\usepackage{hyperref}
\definecolor{darkgreen}{rgb}{0,0.4,0}
\definecolor{BrickRed}{rgb}{0.65,0.08,0}
\hypersetup{colorlinks=true,linkcolor=blue,citecolor=red,filecolor=BrickRed,urlcolor=darkgreen}
\DeclareMathOperator\artanh{artanh}

\newcommand{\LandauO}{\mathcal{O}}
\newcommand{\Landauo}{o}
\newcommand{\PR}{\mathbb{P}} 
\newcommand{\E}{\mathbb{E}} 
\newcommand{\V}{\mathbb{V}} 

\newcommand{\Bc}{\mathcal{B}}

\newcommand{\Nc}{\mathcal{N}}
\newcommand{\Rc}{\mathcal{R}}

\newcommand{\Tc}{\mathcal{T}}

\newtheorem{theo}{Theorem}[section]
\newtheorem{lemma}[theo]{Lemma}

\newtheorem{coro}[theo]{Corollary}

\newenvironment{remark}[1][]{\refstepcounter{theo}\medskip \noindent\textbf{\textit{Remark \thetheo #1:}} }{ \hfill $_{\blacksquare} $\\ }

\newtheoremstyle{conjecture}{}{}{\it}{}{\color{purple}\bfseries}{}{ }{}
\theoremstyle{conjecture}

\graphicspath{{./}{pics/}}
\makeatletter
\def\input@path{{./}{pics}}  
\makeatother

\begin{document}

\author{{Michael Wallner}\\
	 Institute of Statistical Science, \\
	 Academia Sinica, \\
	 Taipei 115, \\
	 Taiwan \\
	 \url{michael.wallner@tuwien.ac.at}
	}
	
\date{}

\title{A bijection of plane increasing trees with relaxed binary trees of right height at most one}

\maketitle
\begin{abstract}
	Plane increasing trees are rooted labeled trees embedded into the plane such that the sequence of labels is increasing on any branch starting at the root. 
	Relaxed binary trees are a subclass of unlabeled directed acyclic graphs. 
	We construct a bijection between these two combinatorial objects and study the therefrom arising connections of certain parameters.
	Furthermore, we show central limit theorems for two statistics on leaves.
	We end the study by considering more than $20$ subclasses and their bijective counterparts. Many of these subclasses are enumerated by known counting sequences, and thus enrich their combinatorial interpretation.

\medskip

\end{abstract}

\newcommand{\OEIS}[1]{\href{http://oeis.org/#1}{OEIS~#1}}
\newcommand{\OEISs}[1]{\href{http://oeis.org/#1}{#1}}

\section{Introduction}
\label{sec:intro}

\let\thefootnote\relax\footnotetext{{\textcopyright}~2018. This manuscript version is made available under the CC-BY-NC-ND 4.0 license \url{http://creativecommons.org/licenses/by-nc-nd/4.0/}}

This paper provides a bijection between a class of directed acyclic graphs (DAGs) shown in Figure~\ref{fig:compacted_trees_n123}, and plane increasing trees  
shown in Figure~\ref{fig:increasing_trees_n123}. 
The number of elements with $n$ nodes is given by the odd double factorials (\OEIS{A001147} \cite{Sloane}) $(2n-1)!! := (2n-1) (2n-3) \cdots 3 \cdot 1$.


\begin{figure}[ht]
	\centering
	\includegraphics[width=0.5\textwidth]{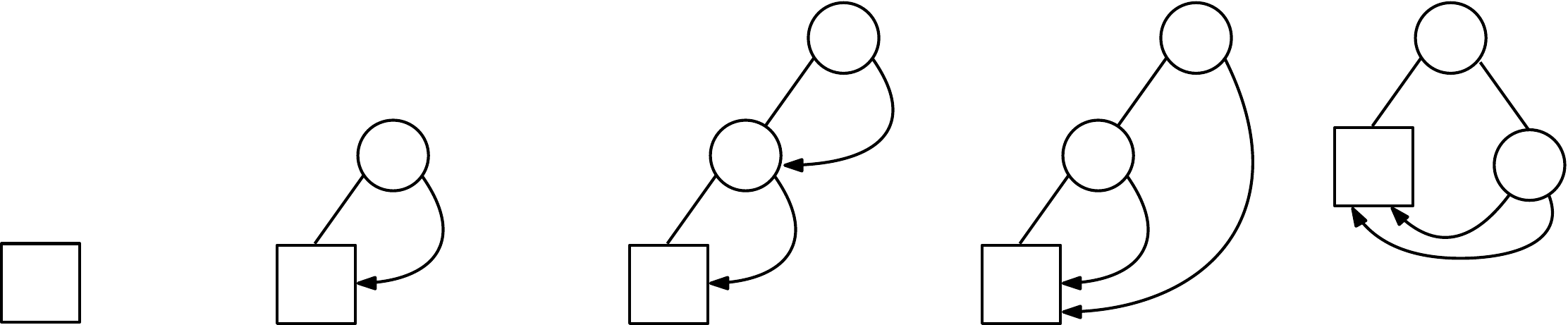}%
	\caption{All relaxed binary trees of size $0,1,2$. Internal nodes are depicted by circles, the unique leaf is depicted by a square. Note that in general these are not trees as there appear directed and undirected edges.}
	\label{fig:compacted_trees_n123}
\end{figure}

We start with some basic definitions. For more details we refer to the excellent book~\cite{drmo09}. 
A \emph{rooted tree} of size $n$ is a connected undirected acyclic graph with $n+1$ nodes, $n$ edges, and a distinguished node called the root. 
All trees appearing in this paper will have a root and we will shortly speak only of trees.
The root introduces an order in the tree given by generations. The root is in generation $0$. All neighbors of the root are in generation $1$, and in general, nodes at distance $k$ from the root are in generation $k$. For an arbitrary node of generation $k>0$ its unique neighbor in generation $k-1$ is called its \emph{parent}. All other neighbors (which are necessarily in generation $k+1$) are called its \emph{children}.

An \emph{increasing tree} is a labeled rooted tree in which labels along any path from the root to the leaves are in increasing order. For notational convenience we label the nodes of a tree from $0$ to $n$ and define its size to be $n$. This concept was first introduced and intensively investigated by Bergeron, Flajolet, and Salvy~\cite{BergeronFlajoletSalvy1992Increasing}. These trees have found vast applications as data structures in computer science, as models in genealogy, and as representations of permutations, to name a few~\cite{drmo09,SmytheMahmoud1994Survey}. 

A tree is called \emph{plane} (or sometimes also \emph{ordered}) if the children are equipped with a left-to-right order. In other words, trees with a different order of the children, are considered to be different trees. For example the two trees in the center of Figure~\ref{fig:increasing_trees_n123} whose roots have two children with labels $1$ and $2$ are considered to be different trees.

This defines the classical family of rooted plane increasing trees, which can be generated uniformly at random by a growth process: start with the root and label $0$. 
At step $i$ there are $2i-1$ possible places to insert node $i$. Choose one uniformly at random. 
Note that at a node with out-degree $d$ there are $d+1$ possible places to insert a new child. 
This idea is known as the Albert-Barab\'asi model~\cite{AlbertBarabasi2002Networks}.
Note that this method gives a way to generate these trees uniformly at random in linear time.

The \emph{degree} of a node is the number of its neighbors, whereas the \emph{out-degree} is the number of its children. 
Nodes of degree $1$ (and therefore out-degree $0$) are called \emph{leaves} or \emph{external nodes}. All other nodes are called \emph{internal nodes}. 
A \emph{young leaf} is a leaf without left sibling.  
A \emph{maximal young leaf} is a young leaf with maximal label, see Figure~\ref{fig:increasing_trees_n123} (see~\cite[Section~4.3]{Callan2009Factorial} for a recurrence relation of plane increasing trees built on this parameter).

\begin{figure}[ht]
	\centering
	~~~~~~
	\includegraphics[width=0.5\textwidth]{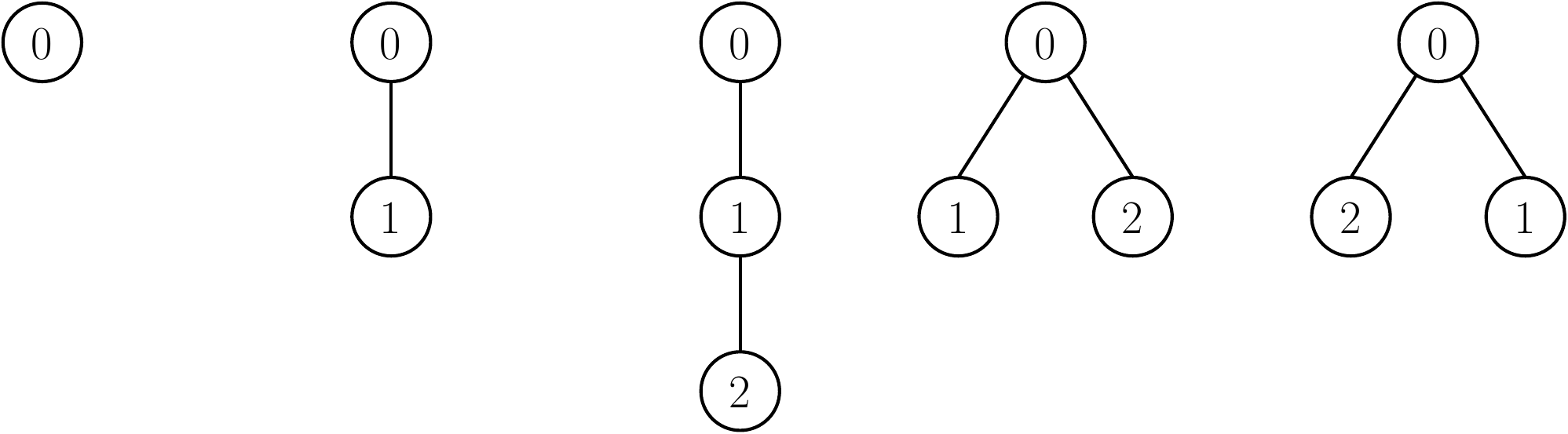}%
	\hfill
	\includegraphics[width=0.18\textwidth]{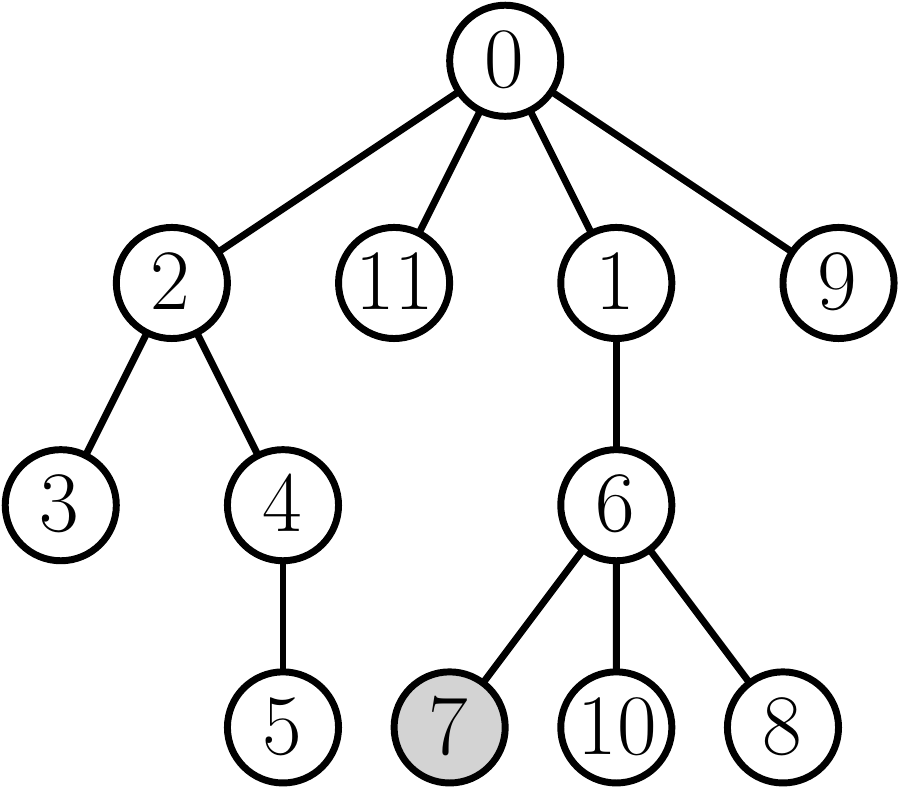}
	~~~~~~
	\caption{Left: All plane increasing trees of size $0,1,2$. Right: An increasing tree of size $11$ with the young leaves~$3,5,7$ and the maximal young leaf~$7$.}
	\label{fig:increasing_trees_n123}
\end{figure}

The second family we are interested in are the less known compacted and relaxed trees. Let us start with their origin and give a definition thereafter. 

In computer science trees are a widely used data structures. Yet real world data often contain vast amount of redundant information.
A strategy to save memory is to store every distinct subtree only once and to mark repeated appearances. 
This concept finds applications in the efficient storage of XML documents~\cite{bousquet2015xml}, and the design and analysis of algorithms and compilers~\cite{flaj09}.
The gain in memory was studied by Flajolet, Sipala, and Steyaert in~\cite{flss90}. 
The corresponding procedure defines a subclass of DAGs, called \emph{compacted trees}, which are in bijection with the original trees, see~\cite{GenitriniGittenbergerKauersWallner2016,flss90}. 
The characterizing property of the generated structure is the uniqueness of each subtree which in the end brings savings in memory. 

We will not need a precise definition of compacted trees, but of a related class, the one of relaxed binary trees.
These appear when the uniqueness condition of subtrees is neglected. 
Let us give a precise definition of this class.

A \emph{relaxed tree} of size $n$ is a directed acyclic graph with $n$ internal nodes, one leaf, $n$ internal edges, and $n$ pointers which is rooted at an internal node. It is constructed from a tree of size $n$, where the first leaf in a postorder traversal is kept and all other leaves are replaced by pointers. These may point to any node that has already been visited in a postorder traversal.
It is called \emph{binary} if it was constructed from a binary tree. 
All relaxed trees considered in this paper will be relaxed binary trees  

In Figure~\ref{fig:compacted_trees_n123} we see all relaxed binary trees of size $0,1$, and $2$. 
Note that for this small sizes all relaxed binary trees are also compacted binary trees. 
However, for size $3$ there are $16$ relaxed binary trees and only $15$ compacted binary trees.

	The \emph{right height} is the maximal number of right edges (or right children) on all paths from the root to any leaf after deleting all pointers. 
	The \emph{level} of a node is the number of right edges on the path from the root to this node, see Figure~\ref{fig:rightheight}. 

\begin{figure}[htb]
	\centering
	\includegraphics[width=0.17\textwidth]{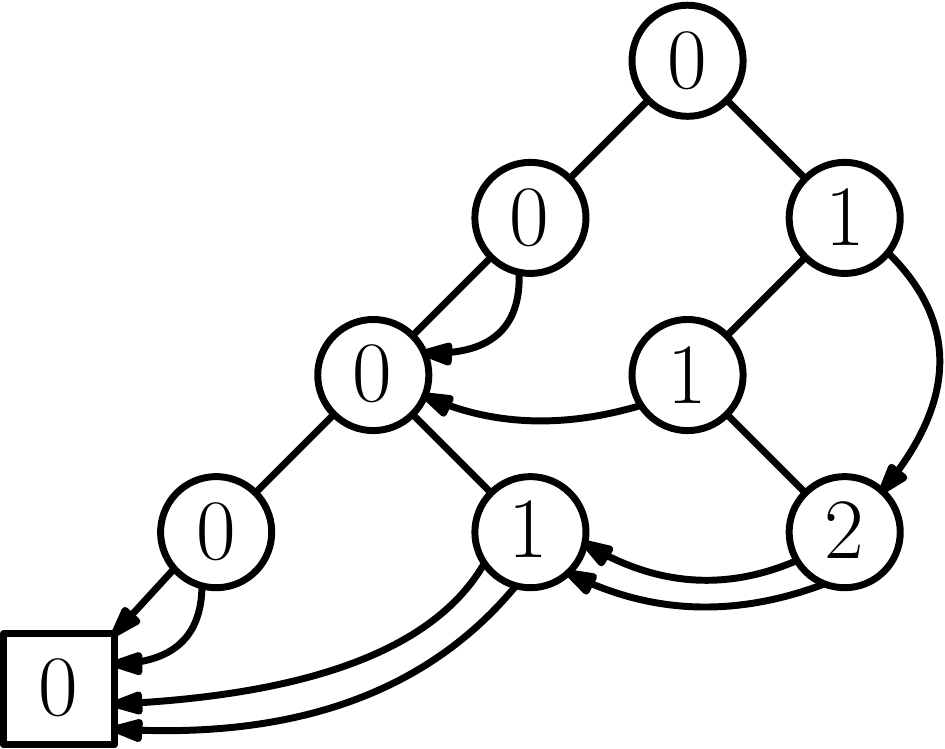}
	\qquad \qquad \qquad
	\includegraphics[width=0.2\textwidth]{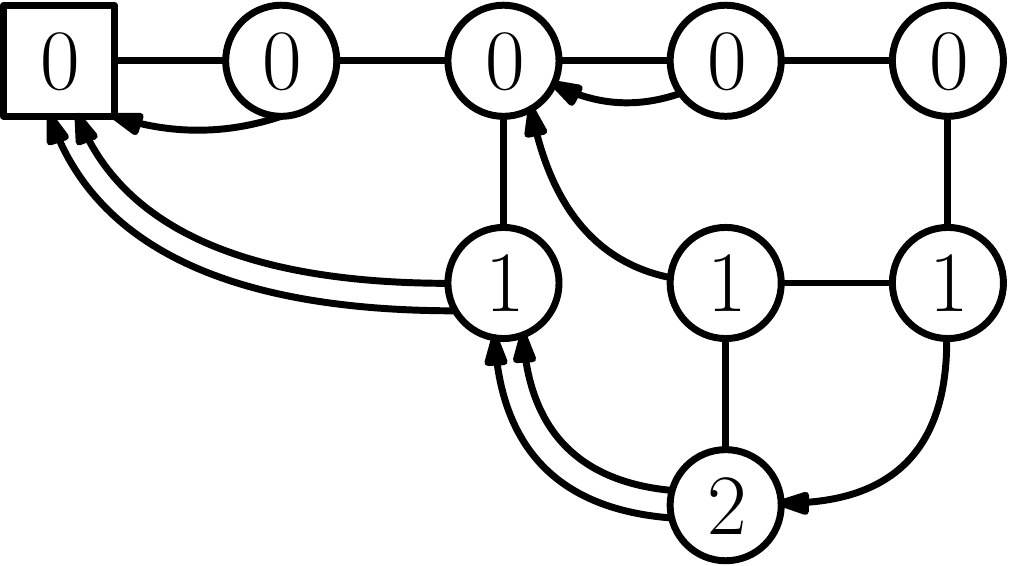}
	\caption{Left: A compacted binary tree with right height~$2$.
	The labels give the level of the node. Right: The same tree rotated by $45$ degrees. The unique leaf is marked by a square.}
	\label{fig:rightheight}
\end{figure}

The asymptotic counting problem for relaxed (and the more restrictive class of compacted) binary trees when restricted to being of finite right height was solved in~\cite{GenitriniGittenbergerKauersWallner2016}. 

\begin{figure}[htb]
	\centering
		\includegraphics[width=0.9\textwidth]{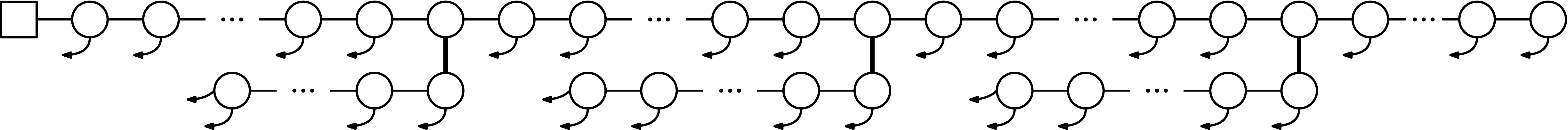}
	\caption{The structure of a relaxed binary tree with right height at most one. For clarity the pointers are only attached to their source. Note that for a specific relaxed tree the pointers are fixed and point to specific nodes seen before the source node in postorder traversal.  }
	\label{fig:R1structure}
\end{figure}

The general structure of relaxed binary trees of right height at most one is shown in Figure~\ref{fig:R1structure}.
In~\cite[Theorem~7.3]{GenitriniGittenbergerKauersWallner2016} it was shown that they admit the exponential generating function
\begin{align}
	R(z) = \frac{1}{\sqrt{1-2z}} = \sum_{n \geq 0} (2n-1)!! \frac{z^n}{n!}.
	\label{eq:R}
\end{align}
In other words, the number of relaxed binary trees of right height at most one of size $n$ is equal to the number of increasing plane trees of size $n$ and is equal to $(2n-1)!!$.
These numbers count more than a dozen labeled objects (see \OEIS{A001147}), yet the class of DAGs is to our knowledge the first not labeled one. 
Bijections appear repeatedly in the literature in order to relate properties of different objects to each other. See for example Janson~\cite{Janson2008Recursive} for a bijection between plane increasing trees and Stirling permutations, or Janson, Kuba, and Panholzer~\cite{JansonKubaPanholzer2011Stirling} for a bijection between plane increasing trees and ternary increasing trees.

{\bf Plan of this article.} 
First, in Section~\ref{sec:bijection}, we present our main contribution: a bijection between relaxed binary trees of right height at most one and plane increasing trees.
As a corollary we get a uniform random sampling algorithm for relaxed binary trees of size $n$ of right height at most one requiring $\LandauO(n)$ steps and $\LandauO(n)$ memory.  
In Section~\ref{sec:parameters}, we consider the bijection from the point of view of relaxed trees. We investigate the number of elements on level~$0$ and the number of branches (or, equivalently, leaves on level~$1$), and map them to parameters of plane increasing trees. Additionally, we show that they admit a central limit theorem.
In Section~\ref{sec:paramInc}, we analyze the bijection from the point of view of plane increasing trees. We collect known results and relate them to relaxed trees.
Finally, in Section~\ref{sec:subclasses}, we investigate more than $20$ subclasses of the relaxed trees under consideration. We derive their generating functions and relate their counting sequences to known and unknown ones of the OEIS. Thereby we find new interpretations of sequences and discover unexpected connections to Fibonacci numbers.

\section{Bijection}
\label{sec:bijection}

We will need the following concepts: A \emph{branch node} is a node on level~$0$ without pointers to which a branch of nodes on level~$1$ is attached. 
We say that this is the branch node of the nodes in this branch. 
In the Figures~\ref{fig:R1structure} and \ref{fig:R1toPORT} we see three branch nodes each.
A \emph{cherry} is a node with $2$ pointers. 
For a plane increasing tree $\Tc$ we denote by $\Tc_k$ the tree restricted to the labels $0,\ldots,k$. 
For notational convenience, we will speak of relaxed trees always meaning relaxed binary trees.

\begin{algorithm}
	\caption{Relaxed binary tree $\Rc$ $\to$ Plane Increasing Tree $\Tc$}
	\label{alg:rel2inc}
	\begin{algorithmic}[1]
		\State Label nodes of $\Rc$ inorder $v_0,v_1,\ldots,v_n$
		\State For each cherry $v_i$ move left pointer to $v_{i-1}$
			\Comment $v_{i-1}$ is $v_{i}$'s branch node
		\State For each node set $p_i := \text{target of pointer of $v_i$}$
		\If{ $\operatorname{level}(v_i)=1$ and $p_i=v_0$}
			\State $p(v_i) := \text{Branch node of branch of $v_i$}$
		\EndIf
		\State Leaf $v_0 \to$ Root of $\Tc$
		\For{$i$ from $1$ to $n$}
		\If{$\operatorname{level}(v_i)=0$}
		\Comment Parent-pointer
			\State Attach $v_i$ as first child to $p_i$
		\Else
		\Comment Sibling-pointer
			\State Attach $v_i$ as direct sibling right of $p_i$
		\EndIf
		\EndFor
	\end{algorithmic}
\end{algorithm}

The bijection stated below is shown on an example in Figure~\ref{fig:R1toPORT}. From top to bottom and left to right a relaxed binary tree of right height at most one is transformed into a plane increasing tree. Reversing these steps gives the inverse bijection.

Algorithm~\ref{alg:rel2inc} presents a formal description of the  transformation from relaxed binary trees to plane increasing trees.
Let us start with an arbitrary relaxed binary tree of size $n$. 
First, we label the nodes from $0$ to $n$ according to an inorder traversal. 
We use $v_i$ to reference the node with label $i$. 
In the labeling process we ignore pointers. Start at the leaf and label it with~$0$. Then, move to the parent. Whenever we see a node for the first time we attach a label incremented by one. If we meet a branch node we traverse its right branch starting from the cherry from left to right. Then we continue on level~$0$.

Next, we move the first (or left) pointer of each cherry $v_i$ (which has to be on level~$1$) to $v_{i-1}$ which is its branch node due to the previous labeling operation. This operation attaches to each node, except the leaf, a unique pointer.

Then, we separate the pointers into two sets: parent- and sibling-pointers. A \emph{parent-pointer} is any pointer starting on level~$0$, and a \emph{sibling-pointer} is any pointer starting on level~$1$. 

Moreover, every sibling-pointer that points to the leaf~$v_0$ is changed to point to its branch node. This is shown for node~$8$ in Figure~\ref{fig:R1toPORT}.

Finally, we consider the nodes in the order of their labels and build a plane increasing tree. The leaf with label~$0$ becomes the root. If the node has a parent-pointer, we attach it as a first child (very left) of the node it is pointing to. If the node has a sibling-pointer, we attach it as a direct sibling on the right of the node it is pointing to. 

\begin{algorithm}
	\caption{Plane Increasing Tree $\Tc$ $\to$ Relaxed binary tree $\Rc$ }
	\label{alg:inc2rel}
	\begin{algorithmic}[1]
		\State $\Bc := \emptyset$
		\For{$k$ from $0$ to $n$}
			\If{$v_k$ is a maximal young leaf in $\Tc_k$}
				\State\label{attach1} Attach $\Bc$ to current root and move its pointer to last node of $\Bc$ as left pointer 
				\State\label{attach2} Attach $v_k$ as new root with a pointer to the parent of $v_k$ in $\Tc_k$
				\State \label{attach3}$\Bc := \emptyset$
			\Else
				\State \label{attachB} Attach $v_k$ as root to $\Bc$ with a pointer to the left sibling of $v_k$ in $\Tc_k$
			\EndIf
		\EndFor
		\State Perform \ref{attach1}-\ref{attach2}
	\end{algorithmic}
\end{algorithm}

For the reverse bijection we need the notion of young leaves from the introduction. Note that from the previous algorithm, the maximal young leaves are the nodes of level $0$. Its formal description is given in Algorithm~\ref{alg:inc2rel}.

Let us start with an arbitrary plane increasing tree of size $n$.
The tree is traversed iteratively in the order of its labels. 
The algorithm builds the relaxed tree and an auxiliary structure called the branch.
At every step we either extend the tree or the branch, which is on some point attached as right child to a node at level $0$. At the beginning this branch is empty. 

For a label $k$ one of the following two rules applies: First, if the current node $k$ is a maximal young leaf of $\Tc_k$ then attach the branch to the last node on level $0$, move the pointer of this level $0$ node as left pointer to the last node of the branch, and set the branch to be empty. Then, attach the node $k$ as new root node on level $0$. For the pointer the parent rule applies: set its pointer to the parent of node $k$ in $\Tc_k$. 

Second, if the current node is not a maximal young leaf of $\Tc_k$ then attach the node $k$ as new root to the branch. For the pointer the sibling rule applies: set the pointer to the direct left sibling of node $k$ in $\Tc_k$. In the case that this is the current root at level $0$, set the node to the leaf $0$. 

At the end attach the branch to the current root of level $0$ and move its pointer to the last node in the branch as left pointer.

\begin{figure}[htb]
	\centering
	\includegraphics[width=0.92\textwidth]{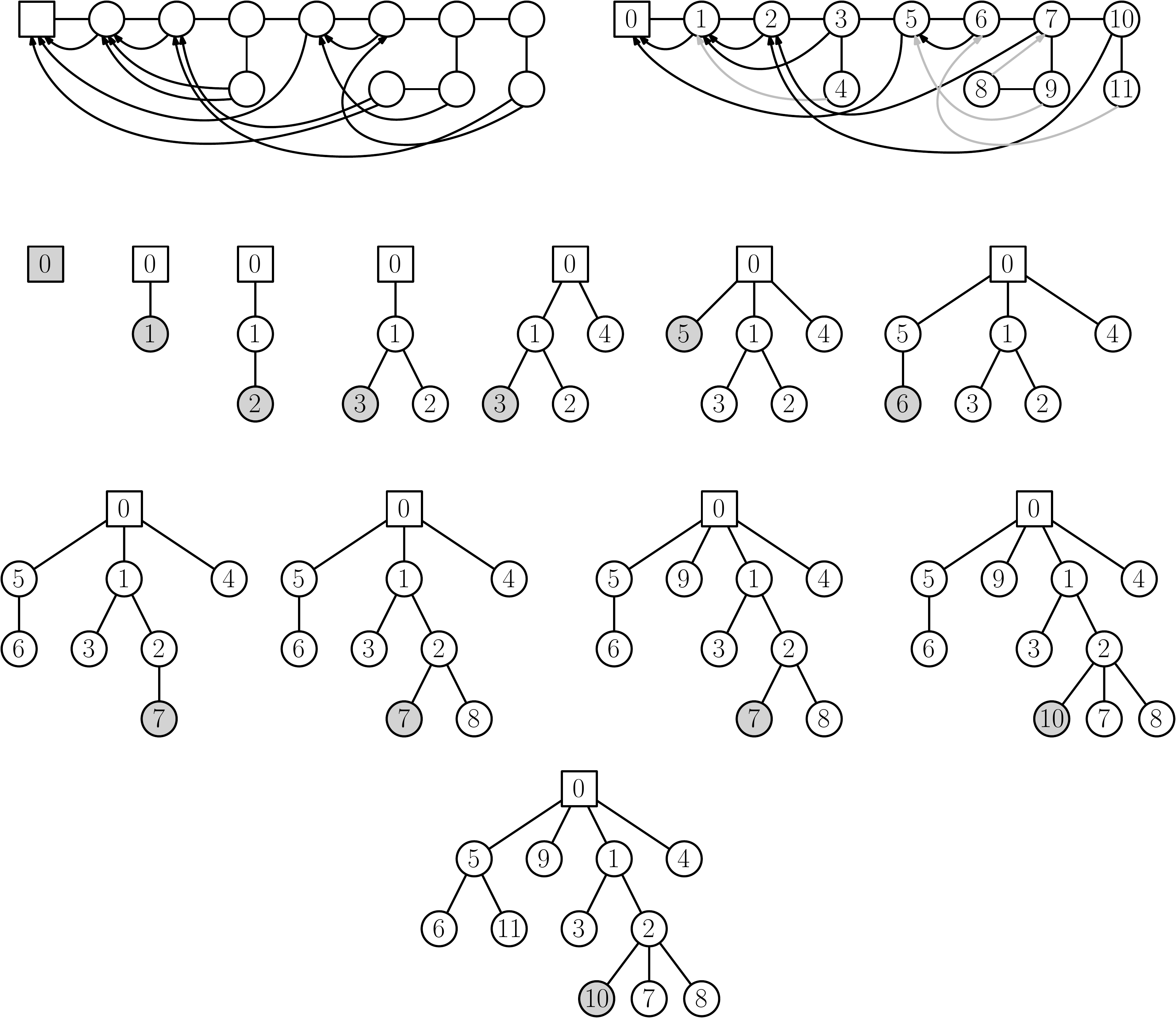}
	\caption{The bijection applied step by step. Parent-pointers are black, and sibling-pointers are gray. The leaf of the relaxed tree marked by a square is transformed into the root of the increasing plane tree. For the reverse bijection the maximal young leaves are shaded in gray.}
	\label{fig:R1toPORT}
\end{figure}

\begin{theo}
	The procedure above is a bijection between relaxed binary trees of right height at most one of size $n$ and plane increasing trees of size $n$. It maps nodes of level $0$ to maximal young leaves in the growth process of the plane increasing tree.
\end{theo}

\begin{proof}
	The procedure uniquely transforms relaxed binary trees of right height at most one of size $n$ into plane increasing trees of size $n$ and vice versa. 
	
	The main observation is the following: On the one hand, when inserting a node into $\Tc_k$ there are $k+1$ places to insert it as maximal young leaf and $k$ not to. On the other hand, when inserting a new node into the relaxed binary tree of size $k$ there are $k+1$ possibilities for the pointer if the node becomes a new root on level $0$, while there are only $k$ possibilities for the pointer if it becomes a new root in the branch. The latter holds, as the pointer cannot point to its (later) branch node. Thus, maximal young leaves correspond to level $0$ nodes and non-maximal leaves to level $1$ nodes.
\end{proof}

\begin{coro}
	Relaxed binary trees of size $n$ of right height at most one can be generated uniformly at random in linear time and with a linear amount of memory.
\end{coro}

\begin{proof}
	The growth process mentioned in the introduction can be used to generate a rooted increasing tree of size~$n$ in linear time using a  linear amount of memory (compare with the Albert-Barab\'asi model~\cite{AlbertBarabasi2002Networks}). Then, Algorithm~\ref{alg:inc2rel} transforms this tree into a relaxed binary tree of size~$n$ with right height at most one in~$n$ steps.
\end{proof}

\begin{remark}
	Note that it is possible to directly generate the relaxed tree of size $n$ by using a growth process for relaxed trees with the ideas of Algorithm~\ref{alg:inc2rel}. Basically, at every point one decides to either attach a new root at level $0$ or in the branch $\Bc$ (which corresponds to level $1$). In the first case one performs operations~\ref{attach1}-\ref{attach3}, and in the second case operation~\ref{attachB}. 
	
	We want to point out that generalizing this method with nested branches it may be used to generate relaxed binary trees with arbitrary or even without height restrictions. However, for the cases of right height larger than $1$ this does not generate them uniformly at random.
\end{remark}

Plane increasing trees are well-studied objects and many statistics exist on their parameters. This bijection transforms some of them into interesting quantities on relaxed binary trees of right height at most one. But vice versa it also leads to interesting results on plane increasing trees. In the next section we consider parameters which are easy to analyze on relaxed trees, and in the section thereafter we look at known results for plane increasing trees.

\section{Parameters of relaxed binary trees}
\label{sec:parameters}

We will use the bivariate generating function $R(z,u) = \sum_{n,k \geq 0} r_{nk} \frac{z^n}{n!} u^k$ with $r_{nk} \geq 0$. 
It is connected to the original generating function by $R(z,1) = R(z)$.  
In particular, for fixed $n$ the sequence $(r_{nk})_{k\geq 0}$  refines the number $r_n$, and we have $\sum_{k\geq 0} r_{nk} = r_n$.
The bivariate generating function $R(z,u)$ will be constructed from the functional equation of $R(z)$ by marking a parameter of interest by an additional variable $u$. For more details of this concept we refer to the excellent book~\cite{flaj09}.

In the sequel we will repeatedly talk about a \emph{sequence of nodes}. This is the \mbox{(sub-)graph} given by a set of internal nodes whose left children are always internal nodes (except maybe the last one) and whose right children are always pointers. Its generic structure is shown in Figure~\ref{fig:chain}.

\begin{figure}[htb]
	\centering
	\includegraphics[width=0.45\textwidth]{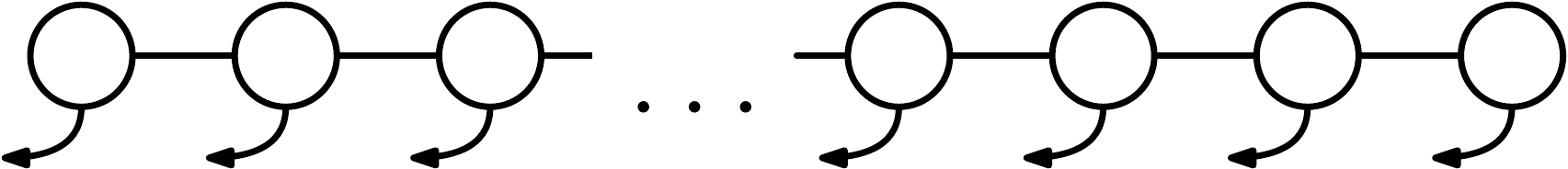}
	\caption{The generic structure of a sequence of nodes. Note that the last left edge, which is omitted here, could either be an internal edge or a pointer.}
	\label{fig:chain}
\end{figure}

Let us therefore briefly revisit the combinatorial construction of $R(z)$ given in~\cite[Corollary~7.2 and Theorem~7.3]{GenitriniGittenbergerKauersWallner2016}. For more details we refer to the deduction in there. The functional equation is equal to
\begin{align}
	R(z) = \frac{1}{1-z} + \frac{1}{1-z} \int{ \frac{1}{1-z} z \left( z R(z)\right)' \, dz}. \label{eq:funceq}
\end{align}
The first term corresponds to the last sequence of nodes on level $0$ after the last branch node. 
It can be interpreted as the initial value or boundary case of the combinatorial construction. 
The factor in front of the integral represents a sequence of nodes on level~$0$ between branch nodes. The integral creates a branch node. The factor $\frac{1}{1-z}$ under the integral creates the nodes of a branch on level~$1$ except the final cherry. Finally, the operator $z \left( z R(z) \right)'$ creates the cherry of the branch. 

Solving this equation, by for example solving the equivalent differential equation, gives the representation of $R(z)$ in~\eqref{eq:R}. In the next subsections we will use this equation by marking certain parameters in order to deduce information on their distribution. For more information on this concept see e.g.,~\cite{flaj09,Wallner16b}. We start with the number of elements on level $0$.

\subsection{Number of elements on level 0 and number of maximal young leaves}
\label{sec:numlvl0}

Let $r_{nk}$ be the number of relaxed binary trees of right height at most one with $k$ internal nodes on level $0$. This is also equal to the number of maximal young leaves in the growth process of a plane increasing tree. Then, the bivariate generating function $R(z,u) = \sum_{n,k \geq 0} r_{nk} z^n u^k$ can be computed from the functional equation~\eqref{eq:funceq} by marking nodes on level~$0$. This gives
\begin{align*}
	R(z,u) = \frac{1}{1-uz} + \frac{u}{1-uz} \int{ \frac{z}{1-z} \frac{\partial}{\partial z} \left( z R(z)\right) \, dz},
\end{align*}
which is then solved to give
\begin{align*}
	R(z,u) &= \frac{1}{(1-(1+u)z)^{\frac{u}{1+u}}}.
\end{align*}
Let $X_n$ be the random variable of the number of internal nodes on level~$0$ of relaxed binary trees with right height at most one drawn uniformly at random among all such trees of size~$n$. Then, we have
\begin{align*}
	\PR(X_n = k) &= \frac{[z^n u^k] R(z,u)}{[z^n] R(z,1)}.
\end{align*}

\begin{theo}
	\label{theo:numlvl0}
	The standardized random variable 
	\begin{align*}
		&\frac{X_n - \mu_1 n}{\sigma_1 \sqrt{n}}, & \text{ with } &&
		\mu_1 &= \frac{1}{2} + \frac{\log(n)}{4n} + \LandauO\left(\frac{1}{n}\right) & \text{ and } &&
		\sigma_1^2 &= \frac{1}{4} - \frac{\pi^2}{32n} + \LandauO\left(\frac{1}{n^2}\right),
	\end{align*}
	converges in law to a standard normal distribution $\Nc(0,1)$.
\end{theo}

\begin{proof}
	The result follows from~\cite[Theorem~4.2]{Sachkov1997Combinatorial} (see also \cite[Theorem~IX.13]{flaj09}), a generalized quasi-powers scheme for bivariate generating functions. The necessary form is proved by the saddle-point method~\cite[Chapter~VIII]{flaj09}. 
\end{proof}

\subsection{Number of branches and number of dominating young leaves}

Recall that a \emph{branch} in a relaxed tree is a sequence of nodes on level $1$. By the bijection these correspond to maximal young leaves, which are not immediately replaced in the growth process by a new young leaf in the next step. We call these \emph{dominating young leaves}. Let $s_{nk}$ be the number of relaxed binary trees of right height at most one with $k$ branches. Then, the bivariate generating function $S(z,u) = \sum_{n,k \geq 0} s_{nk} \frac{z^n}{n!} u^k$ can be computed in a similar way as done in Section~\ref{sec:numlvl0} by marking only the branch node given by the integral. We get
\begin{align*}
	S(z,u) = \frac{1}{\sqrt{1-2z+(1-u)z^2}}.
\end{align*}
Let $Y_n$ be the random variable giving the number of branches of relaxed binary trees with right height at most one of size $n$ drawn uniformly at random among all such trees of size~$n$: 
\begin{align*}
	\PR(Y_n = k) &= \frac{[z^n u^k] S(z,u)}{[z^n] S(z,1)}.
\end{align*}

\begin{theo}
	The standardized random variable 
	\begin{align*}
		&\frac{Y_n - \mu_2 n}{\sigma_2 \sqrt{n}}, & \text{ with } &&
		\mu_2 &= \frac{1}{4} - \frac{1}{8n} + \LandauO\left(\frac{1}{n^2}\right) & \text{ and } &&
		\sigma_2^2 &= \frac{1}{16} + \frac{1}{32n} + \LandauO\left(\frac{1}{n^2}\right),
	\end{align*}
	converges in law to a standard normal distribution $\Nc(0,1)$.
\end{theo}

\begin{proof}
	The result follows the same lines as the one of Theorem~\ref{theo:numlvl0}.
\end{proof}

\section{Parameters of plane increasing trees}
\label{sec:paramInc}

Several parameters of plane increasing trees are well-understood. In order to understand their connection with respect to the stated bijection we introduce the concept of a \emph{pointer-path}. This is a path following only pointers from an arbitrary node to the leaf~$0$ with two special rules: First, due to the transformation of the left cherry pointers to branch nodes, every internal node has exactly one outgoing pointer. Second, if a sibling-pointer points to the leaf it is interpreted as pointing to its branch node, compare node~$8$ in Figure~\ref{fig:pointerpaths}. The length of a pointer-path is given by the number of parent-pointers in it. The results for our stated example are shown in Figure~\ref{fig:pointerpaths}.

These pointer-paths also have an interpretation on the level of increasing trees. Starting from any node, one jumps to its left sibling as long as its label is decreasing. This corresponds to sibling-pointers. If this is not possible any more one moves up to its parent which corresponds to a parent-pointer.  The length of the pointer-path is the depth of the node. In particular, this gives for every node a ``maximal'' decreasing sequence of labels encoded in the tree.  
	
\newcommand*{\papo}{\rightarrow}  
\newcommand*{\sipo}{\,-\,}  
\begin{figure}[h!]
	\begin{minipage}{0.43\textwidth}
	\centering
	\begin{tabular}{|l|c|}
		\hline
		Pointer-path & Length \\
		\hline\hline
		$\phantom{1}1 \papo 0$ & $1$ \\
		$\phantom{1}2 \papo 1 \papo 0$ & $2$ \\
		$\phantom{1}3 \papo 1 \papo 0$ & $2$ \\
		$\phantom{1}4 \sipo 1 \papo 0$ & $1$ \\
		$\phantom{1}5 \papo 0$ & $1$ \\
		$\phantom{1}6 \papo 5 \papo 0$ & $2$ \\
		$\phantom{1}7 \papo 2 \papo 1 \papo 0$ & $3$ \\
		$\phantom{1}8 \sipo 7 \papo 2 \papo 1 \papo 0$ & $3$ \\		
		$\phantom{1}9 \sipo 5 \papo 0$ & $1$ \\
		$10 \papo 2 \papo 1 \papo 0$ & $3$ \\
		$11 \sipo 6 \papo 5 \papo 0$ & $2$ \\
		\hline
	\end{tabular}
	\end{minipage}
	~
	\begin{minipage}{0.55\textwidth}
			\centering
			\includegraphics[width=0.8\textwidth]{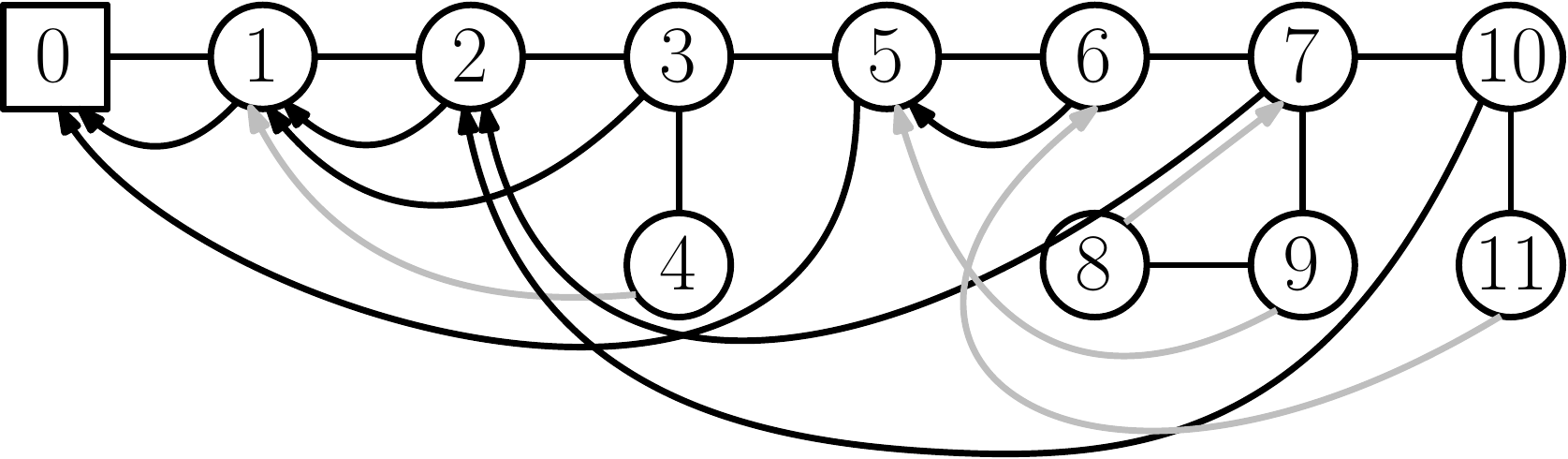}
			
			\bigskip
			\medskip
			
			\includegraphics[width=0.47\textwidth]{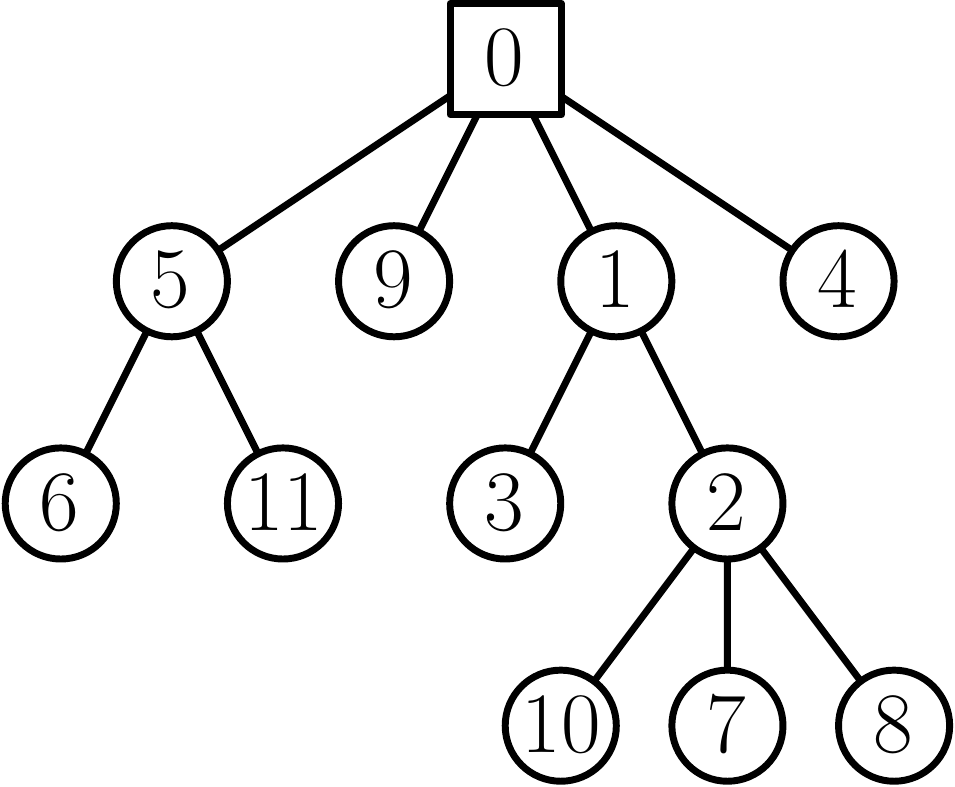}
	\end{minipage}
	\caption{Pointer-paths of the example in Figure~\ref{fig:R1toPORT}. Parent-pointers are marked by~$\papo$ (or black arrows), sibling-pointers are marked by~$\sipo$ (or gray arrows).}
	\label{fig:pointerpaths}
\end{figure}

\begin{table}[h!]
\begin{tabular}{|c|c||c|c|}
	\hline
	Plane increasing tree & Relaxed binary tree & $\E X_n$ & $\V X_n$ \\
	\hline\hline
	\multirow{2}{*}{Depth of node $n$~\cite{Mahmoud1992Distances}} & Length of pointer-path  & \multirow{2}{*}{$\frac{1}{2} \log n + \LandauO(1)$} & \multirow{2}{*}{$\frac{1}{2} \log n + \LandauO(1)$} \\
	& from node $n$ &&\\
	\hline
	\multirow{2}{*}{Number of leaves~\cite{MahmoudSmytheSzymanski1993PORT}} & Number of nodes without  & \multirow{2}{*}{$\frac{2}{3}n + \frac{1}{3}$} & \multirow{2}{*}{$\frac{n}{9} + \frac{1}{18} - \frac{1}{6(2n+1)}$} \\
	& ingoing parent-pointer &&\\
	\hline
	\multirow{2}{*}{Root degree~\cite{BergeronFlajoletSalvy1992Increasing}} & Number of pointer-paths  & \multirow{2}{*}{$\sqrt{\pi n} + \LandauO(1)$} & \multirow{2}{*}{$(4-\pi)n + \LandauO(1)$} \\
	& of length $1$ &&\\
	\hline
	Height~\cite{drmo09, Pittel1994Heights} & Longest pointer-path  & $\frac{1}{2s} \log n + \Landauo(\log n)$ & $\LandauO(1)$\\
	\hline
\end{tabular}
\caption{Parameters of plane increasing trees and the corresponding parameters in relaxed binary trees of right height at most one. The constant $s \approx 0.27846$ is the positive solution of $s e^{s+1}=1$. 
}
\label{tab:param}
\end{table}

There is rich literature on parameters of plane increasing trees, see e.g.~\cite{Mahmoud1992Distances,MahmoudSmytheSzymanski1993PORT,BergeronFlajoletSalvy1992Increasing,drmo09, Pittel1994Heights,JansonKubaPanholzer2011Stirling,KubaPanholzer2007Increasing}. We have summarized four interesting parameters and their counterparts in relaxed binary trees of right height at most one in Table~\ref{tab:param}. 
In the first two cases the standardized random variables $\frac{X_n - \E X_n}{\sqrt{\V X_n}}$ converge in distribution to a standard normal distribution, whereas in the third case the normalized random variable $\frac{X_n}{\sqrt{2n}}$ converges in law to a standard Rayleigh distribution given by the density function $x e^{-x^2/2}$. For details on the distribution of the height see~\cite[Section~6.4]{drmo09} and \cite{Drmota2009Height,BroutiinEtal2008Increasing}.

\begin{remark}
	The Rayleigh distribution in the third case follows directly from the closed form of the number of increasing trees of size $n$ and root degree $k$ given by
	\begin{align*}
		k \cdot \frac{(2n-3-k)!}{2^{n-1-k}(n-1-k)!}.
	\end{align*}
	This was derived in \cite[Corollary~5]{BergeronFlajoletSalvy1992Increasing}, with a small typo of a missing factor~$k$.
\end{remark}

A final interesting parameter is the distribution of out-degrees. Similar to the root degree, the out-degree of a node~$i$ corresponds to the number of pointer-paths of length~$1$ ending with a parent-pointer in~$i$. Note that by definition all pointer-paths ending in $0$ end with a parent-pointer. Let $\lambda_d$ be the limiting probability that a random node has out-degree~$d$. Then, in~\cite{BergeronFlajoletSalvy1992Increasing} it was shown that
\begin{align*}
	\lambda_d &= \frac{4}{(d+1)(d+2)(d+3)}.
\end{align*}
Thus, the probability that a random node has no ingoing parent-pointer is $\frac{2}{3}$, conforming the proportion of number of leaves above. The probability for one ingoing parent-pointer is $\frac{1}{6}$. The case $\lambda_2=\frac{1}{15}$ corresponds to either two parent-pointers whose source nodes do not have sibling-pointers, or one parent-pointer whose source node has exactly one sibling-pointer and this source node has no sibling-pointer.

\section{Subclasses}
\label{sec:subclasses}

At the end we want to consider some subclasses of relaxed binary trees of right height at most one. We will show connections with certain sequences in the OEIS~\cite{Sloane} and solve some open conjectures therein. This adds new combinatorial interpretations to several of them. We start with subclasses that have no initial and/or final sequence of nodes. 

\subsection{Variations of the initial and final sequences}

\begin{figure}[htb]
	\centering
	\includegraphics[width=0.25\textwidth]{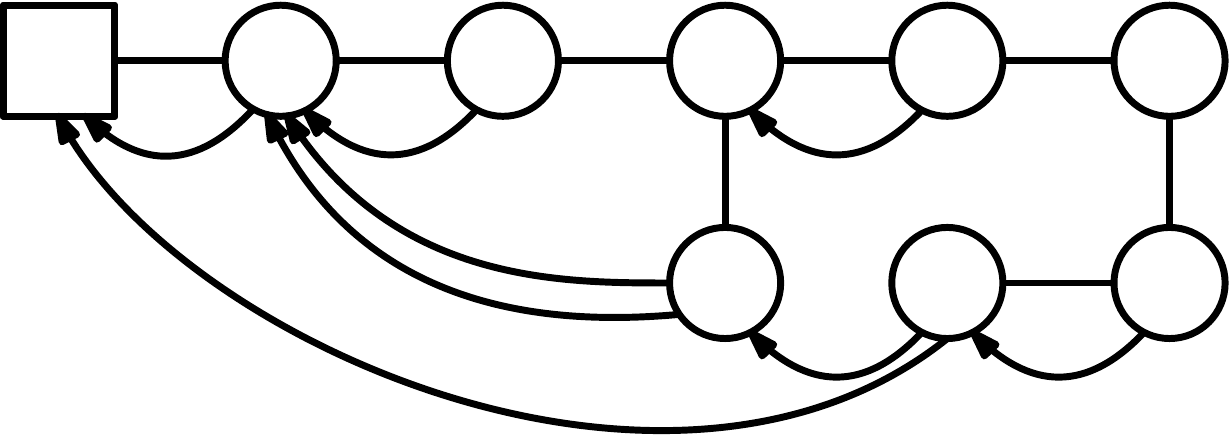}
	\qquad
	\includegraphics[width=0.27\textwidth]{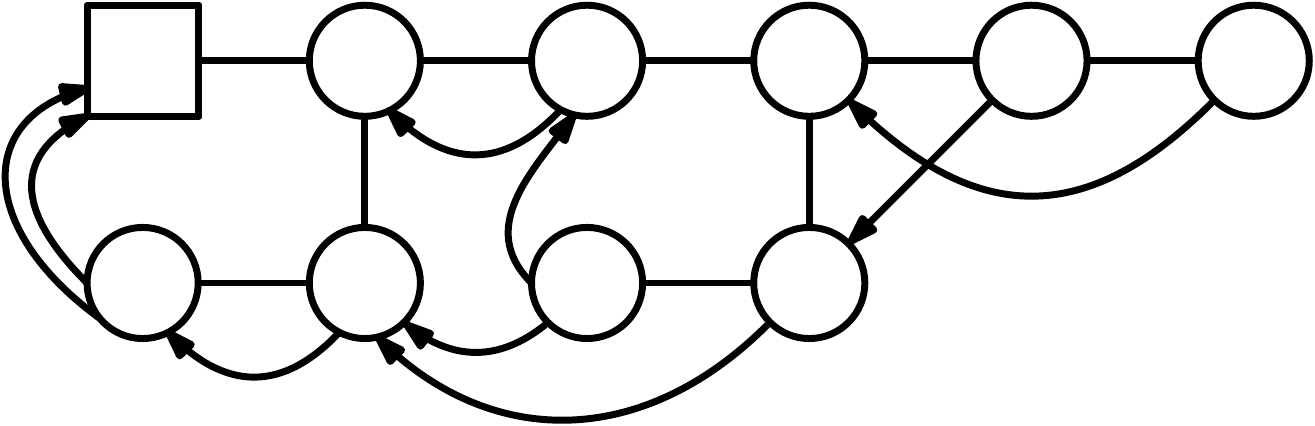}
	\qquad
	\includegraphics[width=0.23\textwidth]{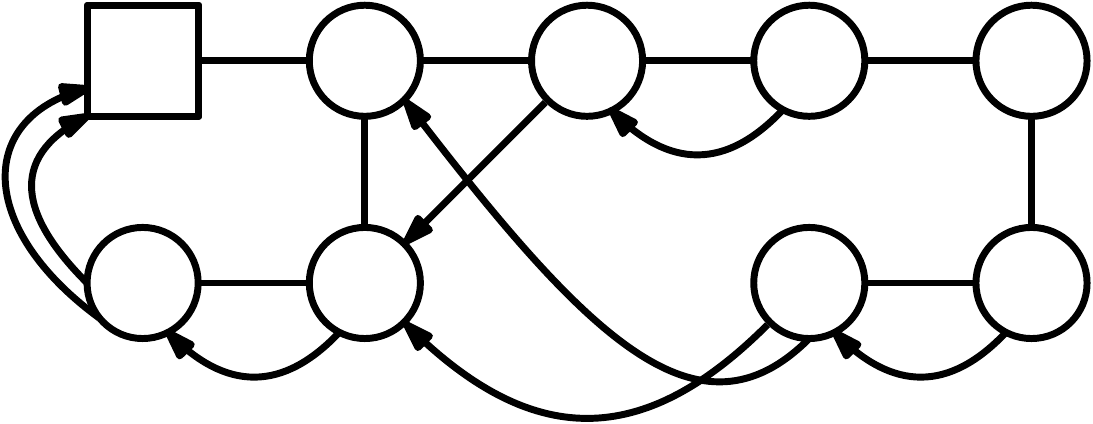}
	\caption{Left: Subclass $\Rc_1$ without initial sequence; Center: $\Rc_2$ without final sequence; Right: $\Rc_3$ without initial and final sequence.}
	\label{fig:S1-3}
\end{figure}

First, we consider the case of an empty initial sequence, see Figure~\ref{fig:S1-3}. In other words all such relaxed trees start with a branch node. 
By~\cite[Proposition~6.4]{GenitriniGittenbergerKauersWallner2016} a multiplication by $1-z$ of the generating function $R(z)$ gives the generating function of this class:
\begin{align*}
	R_1(z) &:= \frac{1-z}{\sqrt{1-2z}} = 1 + \sum_{n \geq 2} (n-1) (2n-3)!! \frac{z^n}{n!} \\
	       &= 1  + \frac{z^2}{2!} + 6 \frac{z^3}{3!} + 45 \frac{z^4}{4!} + 420 \frac{z^5}{5!} + 4725 \frac{z^6}{6!} + \ldots.
\end{align*}
The sequence of coefficients is~\OEIS{A001879} and counts the number of descents in all fixed-point-free involutions of $\{1,2,...,2(n-1)\}$ (we have a shift of minus two). Comparing these numbers to the total number $(2n-1)!!$ of relaxed binary trees of right height at most one, we see that for large~$n$ half of all trees fall into this class.

The bijection transforms this class into the one of plane increasing trees where the leaf with the highest label is not a maximal young leaf, except for the tree of size $0$. Considering these trees we can give an alternative proof of the counting sequence $(n-1) (2n-3)!!$, $n\geq 2$: There are $(2n-3)!!$ trees of size $n-1$ in which we may insert the leaf with label $n$ at $n-1$ out of the $2n-1$ possible places in order not to create a maximal young leaf. 

Second, we consider the related subclass of relaxed binary trees of right height at most one where the final sequence on level $0$ after the last branch node consists of only a single leaf, see Figure~\ref{fig:S1-3}. If there is no branch node then only the leaf belongs to this class. 
From the explanations at the beginning of Section~\ref{sec:parameters} we know that the final sequence corresponds to the first term $\frac{1}{1-z}$ in the functional equation~\eqref{eq:funceq}. Thus, omitting this one and solving the corresponding equation gives the generating function
\begin{align*}
	R_2(z) &:= \frac{1}{3\sqrt{1-2z}} + \frac{2}{3} - \frac{z}{3} 
	         = \sum_{n \geq 0} \frac{(2n-1)!!}{3} \frac{z^n}{n!} \\
	       & = 1  + \frac{z^2}{2!} + 5 \frac{z^3}{3!} + 35 \frac{z^4}{4!} + 315 \frac{z^5}{5!} + 3465 \frac{z^6}{6!} + \ldots.
\end{align*}
This sequence is~\OEIS{A051577} and has no combinatorial interpretation so far. Note that $R_2''(z) = (1-2z)^{-5/2}$. We see that exactly one third of all trees have an empty final sequence.

Trees of this class correspond to plane increasing trees where node $2$ is at depth $1$ and right of node~$1$. As above, we can give an alternative proof of the counting sequence. In particular, after~$2$ steps of the growth process we have a tree with root~$0$ and a single child~$1$. Among the three possible places to insert node~$2$ only one puts it right of node~$1$. Inserting more nodes will not change the relative position of nodes~$1$ and $2$ at depth~$1$. 

As a third class, we look at the combination of the last two classes.  It is given by
\begin{align*}
	R_3(z) &:= (1-z)\left(R_2(z)-1\right) + 1 
	       = 1  + \frac{z^2}{2!} + 2 \frac{z^3}{3!} + 15 \frac{z^4}{4!} + 140 \frac{z^5}{5!} + 1575 \frac{z^6}{6!} + \ldots.
\end{align*}
The sequence of coefficients gives the new entry~\OEIS{A288950}.

\subsection{Trees without sequences -- Connections with Fibonacci numbers}

\begin{figure}[htb]
	\centering
	\includegraphics[width=0.26\textwidth]{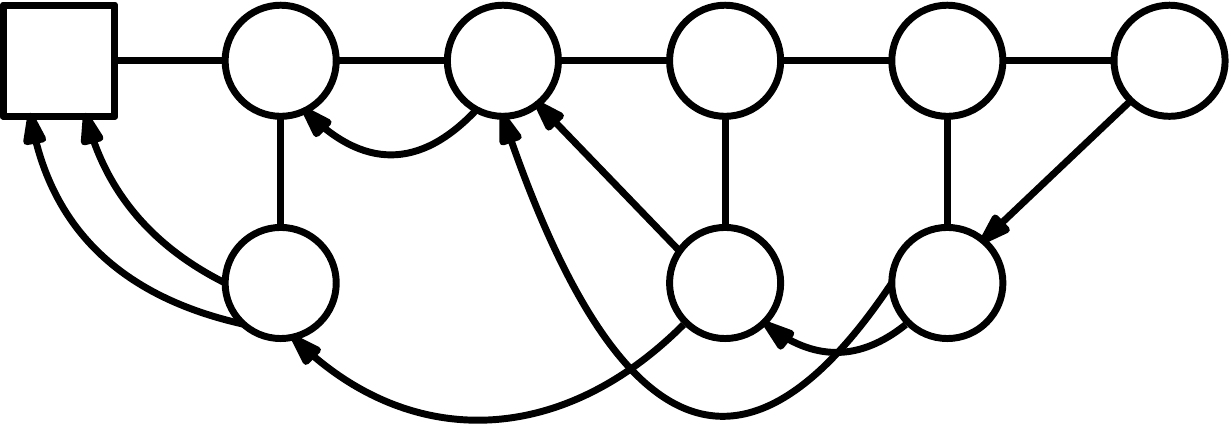}
	\qquad
	\includegraphics[width=0.28\textwidth]{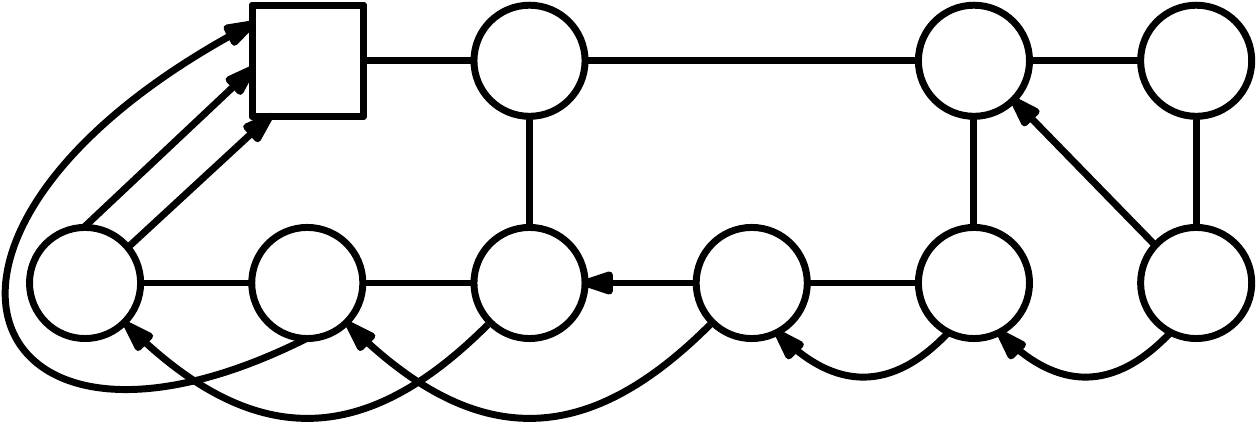}
	\qquad
	\includegraphics[width=0.22\textwidth]{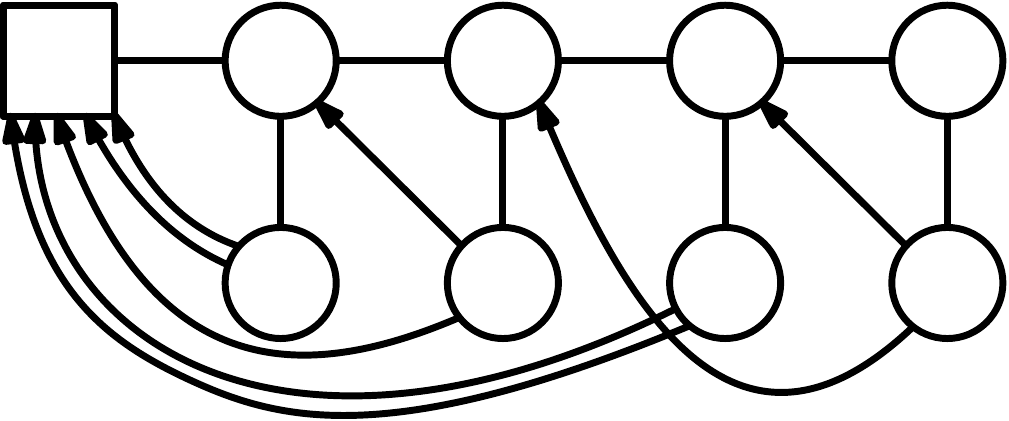}
	\caption{Left: Subclass $\Rc_4$ with at most one node per branch (i.e.,~on level~$1$); Center: $\Rc_5$ without sequences on level~$0$; Right: $\Rc_6$ is the intersection of $\Rc_4$ and $\Rc_5$.}
	\label{fig:S4-6}
\end{figure}

Fourth, let us consider relaxed trees where every sequence on level~$1$ consists of only one element, see Figure~\ref{fig:S4-6}. Adapting the functional equation~\eqref{eq:funceq} we see that the corresponding generating function~$R_4(z)$ satisfies
\begin{align}
	R_4(z) &= \frac{1}{1-z} + \frac{1}{1-z} \int{ z \left( z R(z)\right)' \, dz}, \label{eq:funceqR4}
\end{align}
because the factor $\frac{1}{1-z}$ under the integral would create these sequences. Solving this equation with e.g.,~a computer algebra system like Maple gives
\begin{align*}
	R_4(z) &= \frac{\exp\left(\frac{1}{\sqrt{5}} \artanh\left(\frac{\sqrt{5}z}{2-z}\right)\right)}{\sqrt{1-z-z^2}} 
	        = \frac{1}{\sqrt{1-z-z^2}} \left( \frac{\sqrt{5}+1+2z}{\sqrt{5}-1-2z}\right)^{\frac{\sqrt{5}}{10}} \\
				 &= 1  + z + 3\frac{z^2}{2!} + 13 \frac{z^3}{3!} + 79 \frac{z^4}{4!} + 603 \frac{z^5}{5!} + 5593 \frac{z^6}{6!} + \ldots.
\end{align*}
The second expression is computed by the expression of the $\artanh$ function in terms of logarithms.
This sequence is~\OEIS{A213527}. It implies a different representation. 

\begin{lemma}
\label{lem:fib1}
Let $F_n$ be the $n$-th Fibonacci number, given by $F_0=0,~F_1=1$ and $F_{n} = F_{n-1} + F_{n-2}$ for $n \geq 2$. Then, we have
\begin{align*}
	R_4(z) &= \exp\left(\sum_{n \geq 1} \frac{F_{n+1} z^n}{n} \right) = \frac{1}{1-z-z^2} \exp\left(-\sum_{n \geq 1} \frac{F_{n-1} z^n}{n} \right).
\end{align*}
\end{lemma}

\begin{proof}
	On the one hand we differentiate $G(z) := \sum_{n \geq 1} \frac{F_{n+1} z^n}{n}$ and get $G'(z) = \frac{1+z}{1-z-z^2}$. On the other hand we get from~\eqref{eq:funceqR4} that the logarithmic derivative of $R_4(z)$ is also equal to the same expression. Comparing the initial conditions we deduce that $R_4(z) = \exp(G(z))$.
	
	For the second expression note that $F_{n-1}+F_{n+1} = L_{n}$ which is the $n$-th Lucas number, \OEIS{A000032}. They are defined by $L_0=2,~L_1=1$ and $L_{n} = L_{n-1} + L_{n-2}$ for $n \geq 2$. Furthermore, integrating the known representation $\sum_{n \geq 0} L_{n+1} z^n = \frac{1+2z}{1-z-z^2}$ gives $\sum_{n \geq 1} {\frac{L_n z^n}{n} z^n} = \log\left(\frac{1}{1-z-z^2}\right)$. This proves the claim.
\end{proof}

These relaxed trees correspond bijectively to plane increasing trees where during the growth process never two non-maximal young leaves are inserted after each other. In other words, if $k$ was not a maximal young leaf, then $k+1$ has to be one.

Finally, note that adding constraints like not allowing an initial sequence, not allowing a final sequence, and the combination of both does not lead to any known sequences in the OEIS nor to nice expressions for the generating functions.

As a fifth class, we consider the conjugate class with no sequences between branch nodes on level~$0$, see Figure~\ref{fig:S4-6}. These objects are strongly related to the previous ones. We get
\begin{align*}
	R_{5}(z) &= \frac{\exp\left(-\frac{1}{\sqrt{5}} \artanh\left(\frac{\sqrt{5}z}{2-z}\right)\right)}{\sqrt{1-z-z^2}} 
	        = \frac{1}{\sqrt{1-z-z^2}} \left( \frac{\sqrt{5}-1-2z}{\sqrt{5}+1+2z}\right)^{\frac{\sqrt{5}}{10}} \\
				 &= 1  + \frac{z^2}{2!} + 2 \frac{z^3}{3!} + 15 \frac{z^4}{4!} + 92 \frac{z^5}{5!} + 835 \frac{z^6}{6!} + \ldots.
\end{align*}

This sequence was so far not known in the OEIS. It is now given by \OEIS{A288952}.

\begin{lemma}
Let $F_n$ be the Fibonacci number defined as in Lemma~\ref{lem:fib1}. Then, we have
\begin{align*}
	R_5(z) &= \exp\left(-\sum_{n \geq 1} \frac{F_{n-1} z^n}{n} \right).
\end{align*}
\end{lemma}

\begin{proof}	
	From the closed-form expressions we get the relation $R_4(z) R_{5}(z) = \frac{1}{1-z-z^2}$. Together with the second representation of $R_4(z)$ in Lemma~\ref{lem:fib1} this proves the claim.
\end{proof}

The corresponding plane increasing trees are such that a maximal young leaf has to be followed by a non-maximal young leaf.

Sixth, let us consider a further restriction of the previous class by also not allowing any sequences on level~$0$, see Figure~\ref{fig:S4-6}. This class can be considered maximal with respect to its branches per node. Its functional equation is obtained from~\eqref{eq:funceqR4} by replacing both terms $\frac{1}{1-z}$ by $1$. Then, we get
\begin{align*}
	R_6(z) &:= \frac{1}{\sqrt{1-z^2}} = \sum_{n \geq 0} ((2n-1)!!)^2 \frac{z^{2n}}{(2n)!} \\
	       &= 1  + \frac{z^2}{2!} + 9 \frac{z^4}{4!} + 225 \frac{z^6}{6!} + 11025 \frac{z^8}{8!} + \ldots.
\end{align*}
This sequence is~\OEIS{A177145}. We have $R_6(z) = \sum_{n \geq 0} r_{6,n} \frac{z^n}{n!} = \arcsin'(z)$. Here it is easy to derive the counting formula directly: The only element of size~$0$ is the leaf, $r_{6,0} = 1$. To an element of size $2n$ (which has to be even), we append a branch node connected with a node on level~$1$ which has two pointers. These may point to all elements of the existing tree which gives $(2n+1)^2$ possibilities. This gives $r_{6,2n+2} = (2n+1)^2 \cdot r_{6,2n}$.

From the previous consideration it is easy to identify the corresponding plane increasing trees. Their growth process consists of alternating insertions of maximal and non-maximal young leaves. 

\begin{figure}[htb]
	\centering
	\includegraphics[width=0.3\textwidth]{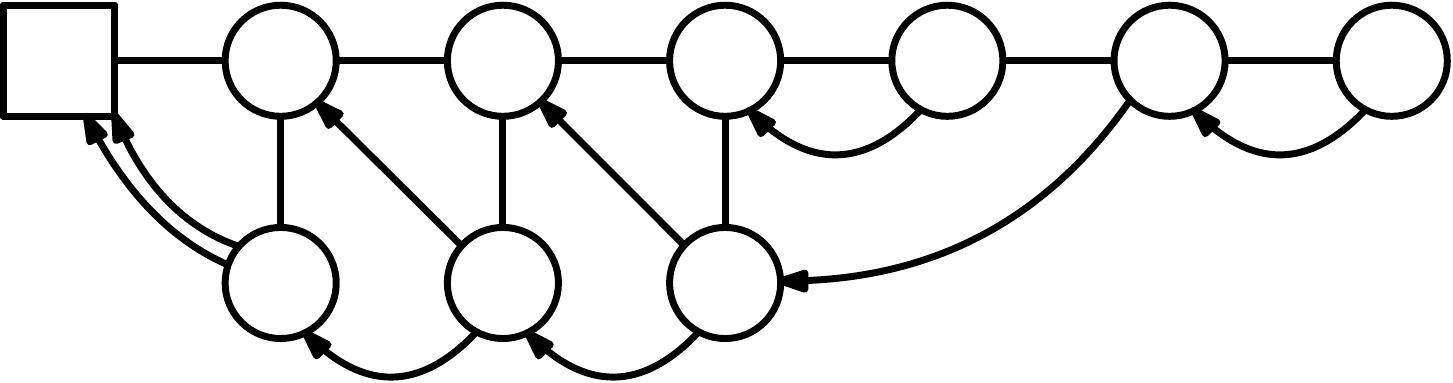}
	\qquad
	\includegraphics[width=0.24\textwidth]{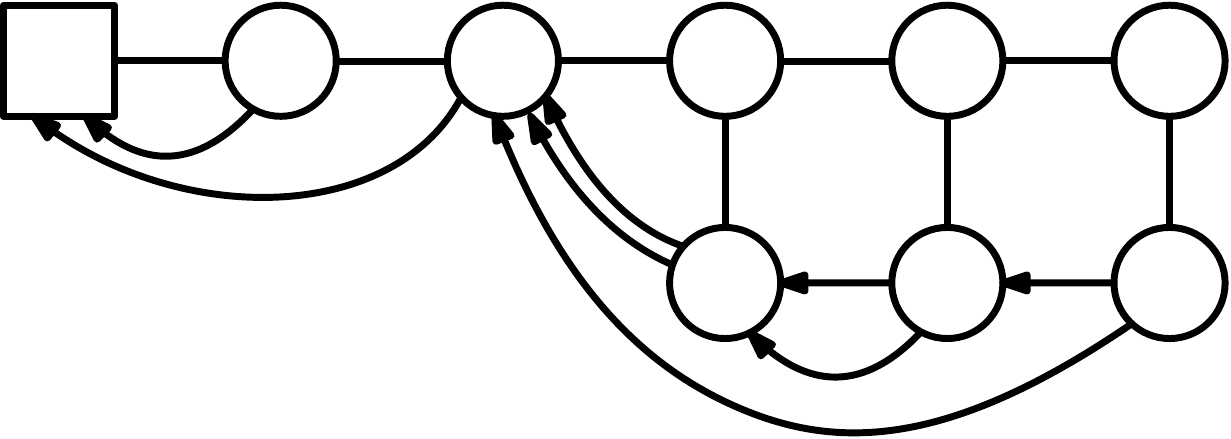}
	\qquad
	\includegraphics[width=0.26\textwidth]{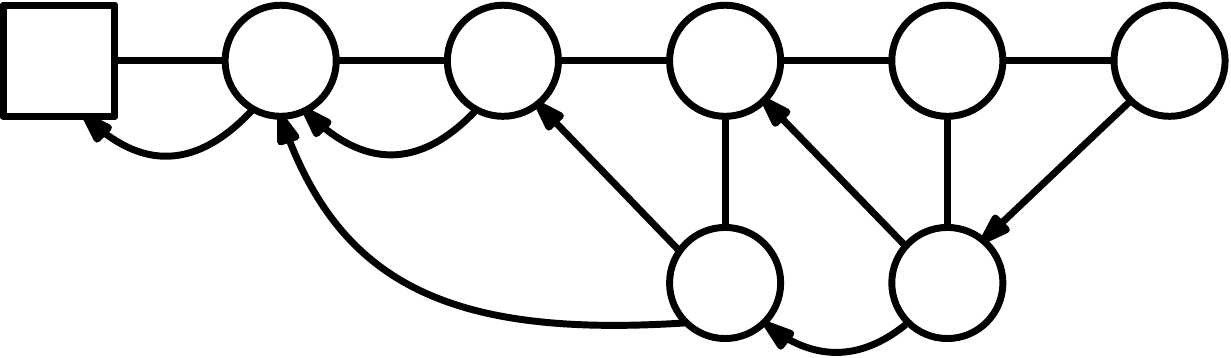}
	\caption{Left: Subclass $\Rc_7$ is like $\Rc_6$ with a possible initial sequence on level~$0$; Center: $\Rc_8$ is like $\Rc_6$ with a possible final sequence on level~$0$; Right: $\Rc_9$ is like $\Rc_6$ with a possible initial and final sequence on level~$0$.}
	\label{fig:S7-9}
\end{figure}

Seventh, we consider a variation of the previous class by allowing an initial sequence on level~$0$, see Figure~\ref{fig:S7-9}. This corresponds to a multiplication of $R_6(z)$ by $\frac{1}{1-z}$ and we get
\begin{align*}
	R_7(z) &:= \frac{1}{1-z}\frac{1}{\sqrt{1-z^2}} 
	        = 1  + z +  3\frac{z^2}{2!} + 9 \frac{z^3}{3!} + 45 \frac{z^4}{4!} + 225 \frac{z^5}{5!} + 1575 \frac{z^6}{6!} + \ldots.
\end{align*}
This sequence is~\OEIS{A000246} and counts the number of permutations in the symmetric group $S_n$ that have odd order. The equivalent class of plane increasing trees is like the previous one except that we allow a sequence of maximal young leaves at the end of the growth process. In other words the consecutive labels $k,\ldots,n$ may be maximal young leaves.

Eighth, we consider the analogous variation of allowing a sequence only at the end of level~$0$, see Figure~\ref{fig:S7-9}. The generating function $R_8(z)$ of this class is obtained by omitting only the factor $\frac{1}{1-z}$ in front of the integral in~\eqref{eq:funceqR4}. This gives
\begin{align*}
	R_8(z) &:= \frac{1}{3\sqrt{1-z^2}} - \frac{z-2}{3(1-z)^2}
	        = 1  + z +  3\frac{z^2}{2!} + 10 \frac{z^3}{3!} + 51 \frac{z^4}{4!} + 280 \frac{z^5}{5!} + 1995 \frac{z^6}{6!} + \ldots.
\end{align*}
This sequence gives rise to the new entry~\OEIS{A288953}. Again, the equivalent plane increasing trees are like the one of case~$6$ but with a possible sequence of maximal young leaves at the beginning of the growth process, i.e., the consecutive labels $1,\ldots,k$ may be maximal young leaves.

Ninth, we consider the combination of the previous two, i.e.,~allowing sequences at the beginning and at the end only on level~$0$, see Figure~\ref{fig:S7-9}. This gives
\begin{align*}
	R_9(z) &:= \frac{1}{3(1-z)\sqrt{1-z^2}} + \frac{3z^3-z^2-2z+2}{3(1+z)(1-z)^3}\\
	       &= 1  + z +  3\frac{z^2}{2!} + 13 \frac{z^3}{3!} + 79 \frac{z^4}{4!} + 555 \frac{z^5}{5!} + 4605 \frac{z^6}{6!} + \ldots.
\end{align*}
This sequence corresponds to the new entry~\OEIS{A288954}. The corresponding plane increasing trees may have consecutive nodes of maximal young leaves $1,\ldots,k$ at the beginning and $\ell,\ldots,n$ at the end. Otherwise maximal and non-maximal leaves alternate.

\subsection{Simplifying the pointer structure}

Tenth, consider the adaption of relaxed trees where both pointers of a cherry are forced to point to the same node (or alternatively the second one is fixed).
The corresponding generating function $R_{10}(z)$ satisfies a functional equation given by~\eqref{eq:funceq} where $(zR(z))'$ is replaced by $R(z)$. The reason is that at the end of the sequence on level $1$ we create only one pointer and let the second one point to the same place. Thus, this subclass is best pictured as the one where cherries have just one pointer. This gives
\begin{align*}
	R_{10}(z) &:= \exp\left(\frac{1}{1-z}\right)
	        = 1  + z +  3\frac{z^2}{2!} + 13 \frac{z^3}{3!} + 73 \frac{z^4}{4!} + 501 \frac{z^5}{5!} + 4051 \frac{z^6}{6!} + \ldots.
\end{align*}
This sequence is~\OEIS{A000262} and counts the number of sets of lists and many other combinatorial objects.

There are many interpretations of the corresponding plane increasing trees. For example a non-maximal young leaf following a maximal young leaf has to be inserted immediately right of it. Or alternatively, as last child of the root. In particular the place of this non-maximal leaf can be chosen uniformly for the class and is fully determined.

Obviously the same subclasses as before can be considered for this class. The $11$ additional results are summarized in Table~\ref{tab:bothsame}.

\begin{table}
\centering
\begin{tabular}{|l|c|l|c|}
	\hline
	Subclass & EGF & Sequence & OEIS \\
	\hline
	\hline
	One cherry pointer & 
	$\exp\left(\frac{z}{1-z}\right)$ & 
	$1, 1, 3, 13, 73, 501, 4051, \ldots$ &
	\OEISs{A000262}\\
	\hline
	No final sequence & 
	Long
	 & 
	$1, 0, 1, 5, 29, 201, 1631, \ldots$ &
	\OEISs{A201203}\\
	\hline
	No initial sequence & 
	$(1-z)\exp\left(\frac{z}{1-z}\right)$ & 
	$1, 0, 1, 4, 21, 136, 1045, \ldots$ &
	\OEISs{A052852}\\
	\hline
	No sequence on level $0$ & 
	$\frac{e^{-z}}{1-z}$ & 
	$1, 0, 1, 2, 9, 44, 265, 1854, \ldots$ &
	\OEISs{A000166}\\
	\hline
	No sequence on level $1$ & 
	$\frac{e^{-z}}{(1-z)^2}$ & 
	$1, 1, 3, 11, 53, 309, 2119, \ldots$ &
	\OEISs{A000255}\\
	\hline
	$+$ no initial sequence & 
	$\frac{e^{-z}}{1-z}$ & 
	$1, 0, 1, 2, 9, 44, 265, 1854, \ldots$ &
	\OEISs{A000166}\\
	\hline
	$+$ no final sequence & 
	$\frac{3e^{-z}+z-2}{(1-z)^2}$ & 
	$1, 0, 1, 3, 15, 87, 597, 4701, \ldots$ &
	\OEISs{A316666}\\
	\hline
	$+$ no initial and final seq. & 
	$\frac{3e^{-z}-z^2}{1-z} - 2$ & 
	$1, 0, 1, 0, 3, 12, 75, 522, \ldots$ &
	\OEISs{A176408}\\
	\hline
	No seq.~on level $0$ and $1$ & 
	$e^{\frac{z^2}{2}}$ & 
	$1, 0, 1, 0, 3, 0, 15, 0, 105, 0, \ldots$ &
	\OEISs{A123023}\\
	\hline
	$+$ initial sequence& 
	$\frac{e^{\frac{z^2}{2}}}{1-z}$ & 
	$1, 1, 3, 9, 39, 195, 1185, \ldots$ &
	\OEISs{A130905}\\
	\hline
	$+$ final sequence & 
	Long
	 & 
	$1, 1, 3, 8, 33, 152, 885, 5952, \ldots$ &
	---\\
	\hline
	$+$ initial and final seq. & 
	Long
	 & 
	$1, 1, 3, 11, 53, 297, 1947, \ldots$ &
	---\\
	\hline
\end{tabular}
\caption{Variations of case~$10$ where each cherry has only one pointer. The comment ``Long'' marks generating functions which do not have a closed form or are too long to state. 
The sequence \OEISs{A316666} is a new entry.
}
\label{tab:bothsame}
\end{table}

\section{Conclusion}

In this paper we provided a bijection between relaxed binary trees (a subclass of directed acyclic graphs arising in the compactification of binary trees) with plane increasing trees.  With the latter being well-studied objects, we had access to a vast amount of results on shape parameters which gave us interesting results on the class of relaxed binary trees. Vice versa we were also able to study new parameters on plane increasing trees, by the corresponding parameters on relaxed binary trees. 
Furthermore, this bijection gave a way to generate relaxed binary trees of size $n$ of right height at most $1$ uniformly at random in linear time using a linear amount of memory.
Finally, we considered more than $20$ subclasses and showed that most of them also enumerate other combinatorial structures. We want to point out that in many cases these are the first non labeled structures.

\section*{Acknowledgment}

We want to thank the referee for the detailed comments which significantly improved the presentation of this work.
This research was partially supported by the Austrian Science Fund (FWF) grant SFB F50-03.

\addcontentsline{toc}{section}{References}
\bibliographystyle{abbrv}
\bibliography{Bibliography}
\label{sec:biblio}

\end{document}